\documentclass[11pt,twoside,a4paper]{amsart}
\usepackage{amssymb}
\usepackage{amsmath,amscd}
\usepackage[initials,nobysame,shortalphabetic]{amsrefs}
\usepackage[UKenglish]{babel}
\usepackage{verbatim}

\usepackage[breaklinks=true, colorlinks=false]{hyperref}
\usepackage[paper=a4paper,left=26mm,right=26mm,top=30mm,bottom=30mm]{geometry}

\makeatletter
\def\section{\@startsection{section}{1}%
  \z@{1.75\linespacing\@plus\linespacing}{1\linespacing}%
  {\normalfont\bf\centering}}
\makeatother

\mathcode`l="8000
\begingroup
\makeatletter
\lccode`\~=`\l
\DeclareMathSymbol{\lsb@l}{\mathalpha}{letters}{`l}
\lowercase{\gdef~{\ifnum\the\mathgroup=\m@ne \ell \else \lsb@l \fi}}%
\endgroup

\newcommand{\ph}{\varphi}\newcommand{\eps}{\varepsilon}
\newcommand{\smath}[1]{\hbox{$#1$}}
\newcommand{\sfrac}[2]{\smath{\frac{#1}{#2}}}
\newcommand{\tpe}{\tilde\ph_\eps}

\newcommand{\minus}{\setminus}
\newcommand{\N}{{\mathbb N}}

\newcommand{\skipit}[2][{\par\noindent[...]\medskip\par}]{\medskip #1}

\RequirePackage{color}\definecolor{RED}{rgb}{1,0,0}\definecolor{BLUE}{rgb}{0,0,1}

\def\la{\langle}
\def\ra{\rangle}

\def\w{\wedge}
\def\dbar{\bar\partial}

\def\C{{\mathbb C}}
\def\P{{\mathbb P}}

\def\Ok{{\mathcal O}}

\DeclareMathOperator{\supp}{supp}

\DeclareMathOperator{\tr}{tr}
\DeclareMathOperator{\codim}{codim}

\def\be{\begin{equation}}
\def\ee{\end{equation}}

\newtheorem{thm}{Theorem}[section]
\newtheorem{lma}[thm]{Lemma}
\newtheorem{cor}[thm]{Corollary}
\newtheorem{prop}[thm]{Proposition}

\theoremstyle{definition}
\newtheorem{df}[thm]{Definition}

\newtheorem{preremark}[thm]{Remark}
\newtheorem{preex}[thm]{Example}

\newenvironment{remark}{\begin{preremark}}{\end{preremark}}
\newenvironment{ex}{\begin{preex}}{\end{preex}}

\numberwithin{equation}{section}

\hypersetup{
  pdftoolbar=true,
  pdfmenubar=true,
  pdftitle={Chern forms of singular metrics on vector bundles},
  pdfauthor={},
  pdfsubject={Preprint},
  pdfkeywords={},
pdfborder= 0 0 .1,
  bookmarksnumbered=false,
}

\begin{document}

\title[]{Chern forms of singular metrics on vector bundles}

\date{\today}

\author{Richard L\"ark\"ang, Hossein Raufi, Jean Ruppenthal, Martin Sera}

\address{R. L\"ark\"ang, H. Raufi, M. Sera, Department of Mathematics\\Chalmers University of Technology and the University of Gothenburg\\412 96 G\"oteborg, Sweden.}

\email{larkang@chalmers.se, raufi@chalmers.se, sera@chalmers.se}

\address{J. Ruppenthal, Department of Mathematics, University of Wuppertal, Gau{\ss}str. 20, 42119 Wuppertal, Germany.}

\email{ruppenthal@math.uni-wuppertal.de}



\thanks{The first author was supported by a grant from the Swedish Research Council, the second author by a grant from the
Olle Engkvist foundation, and the last three authors by the Deutsche Forschungsgemeinschaft
(DFG, German Research Foundation), grant RU 1474/2 within DFG's Emmy Noether Programme.}

\begin{abstract}
We study singular hermitian metrics on holomorphic vector bundles, following Berndtsson-P{\u{a}}un. Previous work by Raufi has shown that for such metrics, it is in general not possible to define the curvature as a current with measure coefficients. In this paper we show that despite this, under appropriate codimension restrictions on the singular set of the metric, it is still possible to define Chern forms as closed currents of order 0 with locally finite mass, which represent the Chern classes of the vector bundle.
\end{abstract}

\maketitle

\section{Introduction}
\noindent Let $X$ be a complex manifold of dimension $n$, let $E\to X$ be a rank $r$ holomorphic vector bundle over $X$, and let $h$ denote a hermitian metric on $E$. The classical differential geometric study of $X$ through $(E,h)$, revolves heavily around the notion of the curvature associated with $h$. This approach requires the metric to be smooth (i.e.\ twice differentiable). However, for line bundles Demailly in \cite{Dem4} introduced the notion of \emph{singular} hermitian metrics, and in a series of influential papers he and others showed how these are a fundamental tool for giving complex algebraic geometry an analytic interpretation.

In \cite{BP} Berndtsson and P{\u{a}}un introduced the following notion of singular metrics for vector bundles:

\begin{df}\label{df:sing}
Let $E\to X$ be a holomorphic vector bundle over a complex manifold $X$. A \emph{singular hermitian metric} $h$ on $E$ is a measurable map from the base space $X$ to the space of hermitian forms on the fibers. The hermitian forms are allowed to take the value $\infty$ at some points in the base (i.e.\ the norm function $\|\xi\|_h$ is a measurable function with values in $[0,\infty]$), but for any fiber $E_x$ the subset $E_0:=\{\xi\in E_x\ ;\ \|\xi\|_{h(x)}<\infty\}$ has to be a linear subspace, and the restriction of the metric to this subspace must be an ordinary hermitian form.
\end{df}

They also defined what it means for these types of metrics to be curved in the sense of Griffiths:

\begin{df}\label{df:Gr}
Let $E\to X$ be a holomorphic vector bundle over a complex manifold $X$ and let $h$ be a singular hermitian metric. We say that $h$ is \emph{Griffiths negative} if $\|u\|^2_h$ is plurisubharmonic for any (local) holomorphic section $u$. Furthermore, we say that $h$ is \emph{Griffiths positive} if the dual metric $h^*$ is Griffiths negative.
\end{df}

\begin{remark}
(i) Strictly speaking, \cite{BP} define $h$ to be Griffiths negative if $\log\|u\|_h$ is plurisubharmonic for any holomorphic section $u$. It is, however, not too difficult to show that these two definitions are equivalent (see e.g. \cite{R}, section 2).\vspace{0.1cm}\\
(ii) Any singular hermitian metric on a vector bundle $E$ induces a dual metric on the dual bundle $E^\ast$ (see Lemma \ref{lma:SingularDual} below). This justifies the notion of Griffiths positivity in Definition \ref{df:Gr} in terms of duality.
\qed
\end{remark}

\noindent Definition~\ref{df:Gr} is very natural as these conditions are well-known equivalent properties for smooth metrics.

Although Definition~\ref{df:sing} is very liberal, as it basically puts no restriction on the metrics, it turns out that Definition~\ref{df:Gr} rules out most of the pathological behaviour. For example, we have the following proposition (\cite{R}, Proposition 1.3 (ii)):

\begin{prop}\label{prop:psh}
Let $h$ be a singular, Griffiths negative, hermitian metric. If $\det h\not\equiv0$, then $i\partial\dbar\log\det h$ is a closed, positive $(1,1)$-current.
\end{prop}

The proof uses the well-known fact that if $h$ is a metric on $E$, then $\det h$ is a metric on $\det E$. For smooth metrics it is also well-known that the curvature of $\det h$, i.e.\ $-\partial\dbar\log\det h = \partial \dbar \log \det h^*$, is the trace of the curvature of $h$, i.e.\ $2\pi i$ times the first Chern form $c_1(E,h)$. Thus, a simple consequence of Proposition~\ref{prop:psh} is that for a singular metric which is curved in the sense of Definition~\ref{df:Gr}, it is possible to define the first Chern form in a meaningful way as a closed, positive or negative $(1,1)$-current.

However, despite this, one of the main results in \cite{R} (Theorem 1.5) is a counter-example that shows that the curvature requirement of Definition~\ref{df:Gr} is not enough to define the curvature of a singular metric as a current with measure coefficients. This rather surprising fact, given the existence of the first Chern form, leads to the question of which differential geometric concepts one can obtain from Definition~\ref{df:Gr}.

The main purpose of this note is to show that under appropriate codimension restrictions on the singular set, it is possible to define Chern forms, $c_k(E,h)$, $k=1,\ldots,\min(r,n)$, associated with a singular metric which is curved as in Definition~\ref{df:Gr}. For metrics that are also continuous outside of the singular set, we can prove the following theorem:

\begin{thm}\label{thm:ContThm}
Let $E\to X$ be a rank $r$ holomorphic vector bundle over a complex manifold $X$, and let $h$ denote a singular, Griffiths positive or negative hermitian metric on $E$. Assume that there is some subvariety $V$ of $X$ with $codim(V)\geq k$ such that $h$ is continuous and non-degenerate outside of $V$.

Then, there exists a unique, closed $(k,k)$-current, $c_k(E,h)$, of order 0 with locally finite mass in $X$ such that for any local regularizing sequence $\{h_\varepsilon\}$ of $h$, with $h_\varepsilon\to h$ locally uniformly outside of $V$, we have that
\begin{equation}\label{eq:weak_conv}
    c_k(E,h_{\varepsilon})\to c_k(E,h)
\end{equation}
in the sense of currents.

More generally, let $E_1,\dots,E_m$ be holomorphic vector bundles on $X$, and for $i=1,\dots,m$, let $h^i$ be a singular Griffiths positive or negative hermitian metric on $E_i$, $k_i \in \mathbb{N}\cup\{0\}$, and let $k := k_1 + \dots + k_m$.
    Assume that there exists a subvariety $V$ of $X$ such that $\codim(V) \geq k$, and such that $h^1,\dots,h^m$ are
    continuous and non-degenerate outside of $V$.

    Then, there exists a unique, closed $(k,k)$-current,
    \begin{equation*}
        c_{k_1}(E_1,h^1) \wedge \dots \wedge c_{k_m}(E_m,h^m),
    \end{equation*}
    of order $0$ with locally finite mass in $X$ such that for any local regularizing sequences $\{h^i_\varepsilon\}$, with $h^i_\varepsilon \to h^i$
    locally uniformly outside of $V$ for $i=1,\dots,m$, we have that
    \begin{equation*}
        c_{k_1}(E_1,h^1_{\varepsilon}) \wedge \dots \wedge c_{k_m}(E_m,h^m_{\varepsilon}) \to
        c_{k_1}(E_1,h^1) \wedge \dots \wedge c_{k_m}(E_m,h^m)
    \end{equation*}
    in the sense of currents.
\end{thm}

\begin{remark}
(i) Let $L$ be a trivial line bundle with a possibly singular positive metric $e^{-\varphi}$, i.e.\ $\varphi$ is a plurisubharmonic function. Let $E$ be the trivial rank $r$ bundle $E = L \oplus \cdots \oplus L$. Equip $E$ with the metric $h$ induced by $e^{-\varphi}$, which is then a Griffiths positive singular hermitian metric. If $\varphi$ is smooth, then $c_k(E,h) = C_k (dd^c \varphi)^k$ for some constant $C_k$. For an arbitrary plurisubharmonic function $\varphi$, one does not have a natural meaning of the product $(dd^c \varphi)^k$, without any condition like for example that the set where $\varphi$ is unbounded is contained in an analytic subset of codimension $k$. Thus, since $h$ degenerates precisely where $\varphi$ is $-\infty$, this shows that the codimension requirements on the degeneracy set of $h$ in Theorem \ref{thm:ContThm} can not be relaxed in general.\vspace{0.1cm}\\
(ii) A singular metric $h$ which is Griffiths positive can always be locally approximated by an increasing sequence $\{h_\varepsilon\}$ of (smooth) Griffiths positive metrics. This is the content of \cite{BP}, Proposition 3.1 and \cite{R}, Proposition 6.2 (see Proposition \ref{prop:appr} below), and the regularizing sequence is obtained through convolution with an approximate identity. Thus, if the singular metric is continuous outside its degeneracy set, these propositions yield that we can always obtain a regularizing sequence which converges locally uniformly on this set.\vspace{0.1cm}\\
(iii) To be precise, from the proof of Theorem~\ref{thm:ContThm} it follows that $c_k(E,h)$ can be written as the difference of two positive, closed $(k,k)$-currents. It is natural to ask whether it is actually positive. However, even when $h$ is a smooth Griffiths positive metric on $E$, it is to the best of our knowledge an open question whether $c_k(E,h)$ is a positive form in general, and only known when $k=1$, which follows from the fact that $c_1(E,h)$ equals the first Chern class of the positive line bundle $(\det E,\det h)$, and when $k=2$ and $\dim X = 2$, which is due to Griffiths, \cite{GrPos} (Appendix to \S 5 (b)). If it is indeed the case that $c_k(E,h)$ is positive for all smooth Griffiths positive metrics on $E$, then it would follow by weak convergence that if $h$ is a singular Griffiths positive metric on $E$, then $c_k(E,h)$ is a positive current.
\end{remark}

\begin{ex}
    There are indeed non-trivial singular metrics on vector bundles satisfying the conditions
    of Theorem~\ref{thm:ContThm}. One such class can for example be found in \cite{Hos}.
    Namely, let $E$ be a vector bundle on a complex manifold $X$. If $E$ has global sections $s_1,\dots,s_m$,
    then these sections induce a morphism $s : E^* \to X \times \C^m$, where $X \times \C^m$ is a trivial rank $m$ bundle,
    which we equip with a trivial metric. Through the morphism $s$, we can define a singular hermitian metric $h^*$ on $E^*$ by
    \begin{equation*}
        \la \xi,\eta \ra_{h^*} := \la s(\xi), s(\eta) \ra.
    \end{equation*}
    This is a singular Griffiths negative metric on $E^*$ since if $\xi$ is a holomorphic section of $E^*$, then $s(\xi)$ is a holomorphic section of the trivial rank $m$ bundle, and hence $\|\xi\|^2_{h^*} = \|s(\xi)\|^2$ is plurisubharmonic.
    We thus obtain a singular Griffiths positive metric $h = h^{**}$ on $E$ (cf. Lemma~\ref{lma:SingularDual}).
    At the points where $s_1,\dots,s_m$ span $E$, the metric $h$ is smooth and non-degenerate,
    and thus, if $s_1,\dots,s_m$ span $E$ on $X \setminus V$ for some subvariety $V$ of $X$ with $\codim(V) \geq k$,
    then the conditions of Theorem~\ref{thm:ContThm} are satisfied to define the Chern current $c_k(E,h)$.
\end{ex}

Using the Chern currents of Theorem \ref{thm:ContThm} it is straightforward to define Chern characters associated with our singular metrics.
For a smooth metric $h$ on a holomorphic vector bundle $E$, the full Chern character can be defined as $ch(E,h) = \tr \exp( (i/2\pi)\Theta(E,h))$.
The Chern character has the advantage that the formula for the Chern character of a tensor product is very simple,
$ch(E\otimes F, h\otimes g) = ch(E,h) \wedge ch(F,g)$. In general, the full Chern character can be expressed as a polynomial in the Chern forms
with integer coefficients.
Since for singular metrics we need to restrict the degree of the products of the Chern forms, we will consider the part $ch_k(E,h)$
of the Chern character of a fixed degree $(k,k)$, which can thus be expressed as:
\begin{equation} \label{eq:ch_k}
    ch_k(E,h) := \tr \left( \frac{(i \Theta)^k}{(2\pi)^k k!} \right) = \sum_K b_K c_{k_1}(E,h) \wedge \dots \wedge c_{k_m}(E,h),
\end{equation}
where $b_K$ are integers, and $K = (k_1,\dots,k_m)$ runs over all partitions of $k$, i.e.\ $k_1,\dots,k_m$ are positive integers and $k_1 + \cdots + k_m = k$.

\begin{df}\label{df:ChernCharacter}
    Let $h$ be a singular, Griffiths positive or negative, hermitian metric on a holomorphic vector bundle $E\to X$. Assume that
    there exists a subvariety $V$ of $X$ with $\codim(V) \geq k$ such that $h$ is continuous and non-degenerate outside of $V$.
    Then the $k$:th \emph{Chern character} $ch_k(E,h)$ is defined by \eqref{eq:ch_k}, where the products of Chern forms in
    the sum in the right-hand side of \eqref{eq:ch_k} have meaning by Theorem~\ref{thm:ContThm}.
\end{df}

The requirement in Theorem~\ref{thm:ContThm} that $h$ should be continuous outside of its degenerate set can be dropped for the existence and uniqueness (see Theorem \ref{thm:main} below), but then we will no longer know that $c_k(E,h_{\varepsilon})$ converges to $c_k(E,h)$. This property is important since an immediate corollary is that basic (local) properties for smooth Chern forms also hold in the singular setting:

\begin{cor}\label{cor:Properties}
Let $X$ be a complex manifold and let $(E,h)\to X$ and $(F,g)\to X$ be two holomorphic vector bundles, where $h$ and $g$ both are singular, Griffiths positive or negative hermitian metrics.

If there exists a subvariety $V$ of $X$ with $codim(V)\geq k$ such that $h$ is continuous and non-degenerate outside of $V$, then:\vspace{0.1cm}\\
(a) The Chern current $c_k(E^*,h^*)$ (where $h^*$ denotes the dual metric on the dual bundle $E^*$) can be defined as in Theorem \ref{thm:ContThm}, and
$$c_k(E^*,h^*)=(-1)^kc_k(E,h).$$
(b) For a complex manifold $Y$ and any holomorphic submersion $f:Y\to X$, the Chern current $c_k(f^*E,f^*h)$ can be defined as in Theorem \ref{thm:ContThm} and
$$c_k(f^*E,f^*h)=f^*c_k(E,h).$$
(c) If there exists a subvariety $V$ of $X$ with $codim(V)\geq k$ such that both $h$ and $g$ are continuous and non-degenerate outside of $V$, then $c_k(E\oplus F,h\oplus g)$ can be defined as in Theorem \ref{thm:ContThm}, and
$$c_k(E\oplus F,h\oplus g)=\sum_{j=0}^k c_j(E,h) \wedge c_{k-j}(F,g),$$
where the products in the sum are defined by Theorem~\ref{thm:ContThm}. \\
(d) If there exists a subvariety $V$ of $X$ with $codim(V)\geq k$ such that both $h$ and $g$ are continuous and non-degenerate outside of $V$, then
$ch_k(E\otimes F,h\otimes g)$ can be defined as in Definition~\ref{df:ChernCharacter}, and
$$ch_k(E\otimes F,h\otimes g)=\sum_{j=0}^k ch_j(E,h) \wedge ch_{k-j}(F,g),$$
where the products in the sum are defined by \eqref{eq:ch_k} and Theorem~\ref{thm:ContThm}.
\end{cor}
\bigskip
\begin{remark}
    By arguing as in \cite{PT}, Lemma 2.3.2 (a), we get that for any complex manifold $Y$, and any holomorphic mapping $f:Y\to X$, the pullback $f^* h$ of a singular, Griffiths positive (negative), hermitian metric $h$, is also Griffiths positive (negative). In Corollary \ref{cor:Properties} (b), we nevertheless need to restrict ourselves to submersions, since otherwise $f^*c_k(E,h)$, the pullback of a \emph{current}, is not well-defined. Also $f$ being a submersion ensures that the codimension requirements needed to define $c_k(f^*E,f^*h)$ are fulfilled.
\end{remark}

As already mentioned, the continuity requirements of Theorem \ref{thm:ContThm} can be relaxed, at the price of losing the weak convergence (\ref{eq:weak_conv}). Thus, for a singular hermitian metric $h$ (under appropriate codimension requirements on the degeneracy set), it is possible to define Chern currents by only requiring the metric to be curved as in Definition~\ref{df:Gr}.

This is achieved by expressing the Chern currents in terms of so called Segre currents. In the smooth setting these are defined as the inverses of the Chern forms (see section 2 below) and can be expressed recursively through,
\begin{equation}\label{eq:recursive}
s_k(E,h)+s_{k-1}(E,h) \wedge c_1(E,h)+\ldots+c_k(E,h)=0,\quad k=1,\ldots,n.
\end{equation}
Hence, any Chern form can be expressed as a sum of products of Segre forms
\begin{equation} \label{eq:ChernSegre}
c_k(E,h) = \sum_K a_K s_{k_1}(E,h) \wedge \dots \wedge s_{k_m}(E,h)
\end{equation}
where $a_K$ are integers, and where $K = (k_1,\dots,k_m)$ runs over all partitions of $k$, i.e.\ $k_1,\dots,k_m$ are positive integers and $k_1 + \cdots + k_m = k$.

Now as already mentioned, a singular Griffiths positive metric $h$ can always be locally approximated by an increasing sequence $\{h_\varepsilon\}$ of Griffiths positive metrics (see Proposition \ref{prop:appr}). The advantage of Segre forms over Chern forms in the singular setting is that we can show that for products of Segre forms with different regularizations, under appropriate codimension restrictions on the singular set of $h$, there exists subsequences such that the iterated limit
\begin{equation}\label{eq:IteratedLimit}
\lim_{\varepsilon^m_\nu \to 0} \cdots \lim_{\varepsilon^1_\nu \to 0} s_{k_1}(E,h_{\varepsilon^1_\nu}) \w \cdots \w s_{k_m}(E,h_{\varepsilon^m_\nu})
\end{equation}
exists in the sense of currents. This limit is furthermore independent of the regularizations, and will thus yield a global current which we will denote by $s_{k_1}(E,h)\wedge\cdots\wedge s_{k_m}(E,h)$. These currents can then be used to define the Chern current $c_k(E,h)$ through (\ref{eq:ChernSegre}). The precise statement of the conditions under which this construction works is as follows.

\begin{thm}\label{thm:main}
Let $E\to X$ be a holomorphic vector bundle over a complex manifold $X$, and let $h$ denote a singular, Griffiths positive, hermitian metric on $E$.

Assume that there is some subvariety $V$ of $X$ with $codim(V)\geq k$ such that $L(\log\det h^*)\subseteq V$, where $L(\log\det h^*)$ denotes the unbounded locus of $\log\det h^*$ (see Definition~\ref{df:ULocus} and Remark~\ref{remark:UEx} below).

Then the $k$:th Chern current of $E$ associated with $h$, $c_k(E,h)$, can be defined through (\ref{eq:ChernSegre}), where the Segre products $s_{k_1}(E,h)\wedge\cdots\wedge s_{k_m}(E,h)$ are defined by the iterated limit (\ref{eq:IteratedLimit}). The Chern current $c_k(E,h)$ will be a closed $(k,k)$-current of order 0 with locally finite mass in $X$.

If $h$ is instead a singular, Griffiths negative, hermitian metric, then the same result holds if $L(\log \det h^*)$ is replaced by $L(\log \det h)$ throughout the statement.
\end{thm}

Although Theorem~\ref{thm:ContThm} is not formulated in terms of Segre forms, its proof also uses the formula \eqref{eq:ChernSegre} for expressing Chern forms of a smooth regularization of the metric in terms of Segre forms. The key difference between the proof of Theorem~\ref{thm:ContThm} and Theorem~\ref{thm:main} is that when $h$ is continuous outside of the non-degeneracy set, then one can use the same regularization in each of the Segre forms and can just take a single limit, while in the setting of Theorem~\ref{thm:main}, one needs to have different regularizations of the metric in each of the Segre forms, and take an iterated limit as in \eqref{eq:IteratedLimit} to see that the limits exist.
\begin{remark}
If the conditions of Theorem~\ref{thm:ContThm} are fulfilled to define the current $c_k(E,h)$, then the conditions of Theorem~\ref{thm:main} are also fulfilled, and one would a priori obtain two different currents $c_k(E,h)$. However, it follows from the respective proofs that these currents indeed coincide.
\end{remark}

Finally, in the smooth setting the most salient feature of Chern forms is of course that they are closed forms whose cohomology classes (in the de Rham cohomology group of smooth forms) are invariants of the vector bundle, i.e.\ independent of the metric. As these forms also define cohomology classes in the de Rham cohomology group of currents, it is natural to ask whether or not the cohomology classes of the Chern currents of Theorem \ref{thm:main} coincide with the usual Chern classes.

Since this cohomology class invariance is a global property, the singular counterpart is not a direct consequence of the weak convergence in either Theorem~\ref{thm:ContThm} or Theorem \ref{thm:main}.
Our next result nevertheless shows that the Chern currents lie in the right cohomology classes when the manifold is compact by using Demailly's regularization in \cite{DemaillyReg}.

\begin{thm}\label{thm:Cohomology}
Let $h$ be a singular hermitian metric on a holomorphic vector bundle $E\to X$ satisfying the assumptions of Theorem \ref{thm:main} so that the Chern current $c_k(E,h)$ can be defined.
If $X$ is compact, then
$$[c_k(E,h)]=c_k(E)$$
where $c_k(E)=[c_k(E,h_0)]$ is the usual Chern class defined by any smooth metric $h_0$ on $E$.
\end{thm}

\begin{remark} Chern classes of certain singular hermitian metrics have been studied earlier by Mumford in \cite{Mum}:
Let $X$ be a Zariski open set of a projective variety $\bar X$ such that $\bar X\minus X$ is a divisor with normal crossings, and let $\bar E$ be a vector bundle on $\bar X$.
In \cite{Mum}*{\S\,1}, Mumford studies smooth hermitian metrics on $E=\bar E|_X$
which satisfy certain growth conditions when approaching $\bar X\minus X$.
These metrics could then be considered as singular metrics on $\bar E$.
Mumford then uses the classical Bott-Chern theory and the growth assumptions to show that trivial extensions of the Chern forms on $X$  to all of $\bar X$ represent the Chern classes of $\bar E$ (in all degrees). 

There is no apparent relation between the growth assumptions in \cite{Mum} and the positivity assumptions that we consider.
Our approach to defining Chern currents is different to the one in \cite{Mum}, where ours is based on using the positivity of the
curvature to be able to use pluripotential theoretical methods in the spirit of Bedford-Taylor-Demailly. This approach
might in particular produce currents with mass on the singular set of the metric, in contrast to the case in \cite{Mum}.
\end{remark}

The paper is organized as follows. Section 2 contains a short introduction to Segre forms. In section 3 we collect a few preliminary results that will be needed in the proofs of Theorem \ref{thm:ContThm} and Theorem \ref{thm:main}, and some basic facts from pluripotential theory. Section 4 is devoted to the proofs of Theorem \ref{thm:ContThm} and Theorem \ref{thm:main}. Finally, in section 5 we end by showing that on a compact manifold, the Chern currents of Theorem \ref{thm:main} are in the right cohomology class. 

\section*{Acknowledgements}
\noindent It is a pleasure to thank Bo Berndtsson and Mihai P{\u{a}}un for expressing interesting in our work and for valuable comments.
We also thank the anonymous referee for helpful suggestions which helped to improve the paper.

\section{Segre forms}
\label{sec:segre-forms}

\noindent Let $E \to X$ be a holomorphic vector bundle with a smooth metric $h$. By definition, the \emph{total Segre form} $s(E,h)=1+s_1(E,h)+\cdots+s_n(E,h)$ associated with $(E,h)$ is the multiplicative inverse of the total Chern form $c(E,h)=1+c_1(E,h)+\cdots+c_n(E,h)$. Hence, as mentioned in the introduction, they can be expressed recursively through,
$$s_k(E,h)+s_{k-1}(E,h)\wedge c_1(E,h)+\cdots+c_k(E,h)=0,\quad k=1,\ldots,n.$$
The first three Segre forms, for example, are
{\setlength\arraycolsep{2pt}
\begin{eqnarray}
s_1(E,h)&=&-c_1(E,h),\nonumber\\
s_2(E,h)&=&c_1(E,h)^2-c_2(E,h),\nonumber\\
s_3(E,h)&=&-c_1(E,h)^3+2c_1(E,h) \wedge c_2(E,h)-c_3(E,h),\nonumber
\end{eqnarray}}
\!\!and expressing the first three Chern forms in terms of these Segre forms yields,
{\setlength\arraycolsep{2pt}
\begin{eqnarray}
c_1(E,h)&=&-s_1(E,h),\nonumber\\
c_2(E,h)&=&s_1(E,h)^2-s_2(E,h),\nonumber\\
c_3(E,h)&=&-s_1(E,h)^3+2s_1(E,h)\wedge s_2(E,h)-s_3(E,h).\nonumber
\end{eqnarray}}

Our interest in the Segre forms stems from the fact that they turn out to be closely related to the projectivized bundle, $\pi:\P(E)\to X$, associated with $E$. Recall that this fiber bundle is constructed by letting $\P(E)_x:=\P(E_x^*)$ for each $x\in X$ (the projectivization of an $r$-dimensional vector space, where $E^*$ denotes the dual bundle of $E$). The pullback bundle $\pi^*E^*\to\P(E)$ will then carry a tautological sub-bundle $\Ok_{\P(E)}(-1)$, where the notation is justified by the fact that fiberwise this is nothing but $\Ok(-1)$ over $\P^{r-1}$. The global holomorphic sections of $\Ok_{\P(E)}(1)$ (the dual of $\Ok_{\P(E)}(-1)$), over any fiber are in one-to-one correspondence with linear forms on $E_x^*$, i.e.\ with the elements of $E_x$ (this is the reason for projectivizing $E^*$ instead of $E$).

Now if $h$ is a smooth metric on $E$, then $\Ok_{\P(E)}(1)$ can be equipped with a natural metric which we will denote by $e^{-\varphi}$. We let $\Phi$ denote the first Chern form of this metric, i.e.
$$\Phi:=\frac{i}{2\pi}\Theta\big(\Ok_{\P(E)}(1),e^{-\varphi}\big).$$
The following proposition is at the center of our interest in Segre forms:
\begin{prop}\label{prop:Segre}
Let $h$ be a smooth hermitian metric on a holomorphic vector bundle $E \to X$.
Then
$$s_k(E,h) = (-1)^k\pi_*\big(\Phi^{k+r-1}\big).$$
\end{prop}

To our knowledge this result first appears in \cite{Mou} (see also \cite{G} and \cite{Diverio}). We will use Proposition~\ref{prop:Segre} to define Segre forms in the singular setting. Before we turn to this we first need to establish a few preliminary results and recall some facts from pluripotential theory.

\section{Preliminaries}
\noindent Throughout the paper, we are interested in singular hermitian metrics that are Griffiths positive as in Definition \ref{df:Gr}. Since this condition is formulated in terms of the dual metric we first of all need to check that this dual metric is also a singular hermitian metric in the sense of Definition \ref{df:sing}. Already in \cite{BP} (section 3) this is claimed in passing, but since no argument is presented, for the sake of completeness, we include a proof here.

\begin{lma}\label{lma:SingularDual}
Any singular hermitian metric $h$ on a vector bundle $E$ induces a canonical dual singular hermitian metric $h^*$ on the dual bundle $E^*$
such that $h^{**} = h$ under the natural isomorphism $E^{**} \cong E$.
\end{lma}

\begin{proof}
As the statement is pointwise, it is enough to consider a fixed fiber $E_x=:V$. Recall that by Definition \ref{df:sing} a singular hermitian form $h$ on $V$ is a map
	\[V\rightarrow [0,\infty],\ \|\xi\|_h:=\begin{cases}\|\xi\|_{h_0}&\hbox{if } \xi\in V_0 \\\  \infty & \hbox{otherwise} \end{cases}\]
where $V_0$ is the subspace of $V$ where $h$ is finite, and $h_0$ is a hermitian form (with values in $[0,\infty)$) on $V_0$.

We want to show that any singular hermitian form $h$ on a vector space $V$ given by the subspace $V_0$ and the hermitian form $h_0$ on $V_0$ induces a singular dual hermitian form on the dual of $V$.

Let $N$ denote the linear subspace of $V_0$ where the hermitian form $h_0$ degenerates, i.\,e., $N$ is the eigenspace of the eigenvalue $0$. Then, $V_0/N$ admits a (canonical) hermitian form $\tilde h_0$ induced by $h_0$ which is non-degenerate (given by $\|\xi+N\|_{ \tilde h_0}:=\|\xi\|_{h_0}$). In particular, there exits a dual hermitian form ${\tilde h_0}{\!\!\big.}^\ast$ on $(V_0/N)^\ast$.
We let $W := V^*$ and $W_0:=N^o=\{\eta\in W:\eta|_N=0\}$ be the annihilator of $N$.
Each $\eta_0:=\eta|_{V_0}$ with $\eta\in W_0$ can be seen as an element $\tilde \eta_0$ of $(V_0/N)^\ast$ defined by $\tilde \eta_0 (\xi+N) := \eta_0(\xi)$,
Hence, we obtain the dual singular hermitian form
	\[\|\eta\|_{h^\ast}:= \begin{cases}\big\|\widetilde{\eta|_{V_0}}\big\|_{{\tilde h_0}{\!\!\!\big.}^\ast} &\hbox{if } \eta\in W_0 \\ \ \infty & \hbox{otherwise}\end{cases} \]
on $W=V^\ast$ associated to $h$.

If we let $M$ denote the linear subspace of $W_0$ where $h^*$ degenerates, then by definition of $h^*$, one obtains that $M = V_0^o=\{ \eta \in W_0 : \eta|_{V_0} = 0 \}$ is the annihilator of $V_0$.
Thus, it follows that $W_0/M \cong (V_0/N)^*$.
If we let $U := W^* \cong V$ and do the above construction for $(W,h^*)$, then one obtains, using the isomorphism $U \cong V^{**}$, that
\begin{equation*}
    U_0 = M^o\cong  \{ v : \eta(v) = 0 \text{ for all $\eta \in W_0$ such that $\eta|_{V_0} = 0$} \} = V_0.
\end{equation*}
Finally if $P$ is the null space of the induced metric $h^{**}$, then
\begin{equation*}
    P=W_0^o \cong \{ v \in V_0 : \eta(v) = 0 \text{ for all $\eta \in W_0$ such that $\eta|_N = 0$} \} = N,
\end{equation*}
and one readily verifies that the metric $h^{**}$ on $U_0 \cong V_0$ equals $h$.
\end{proof}

In the previous section we mentioned the close connection between Segre forms and the Serre line bundle $\Ok_{\P(E)}(1)$ over the projective bundle $\P(E)$ (Proposition \ref{prop:Segre}), and how we will use this connection to define Segre currents in the singular setting. For a smooth metric $h$ it is well-known that if $h$ is Griffiths positive, then the induced metric $e^{-\varphi}$ on $\Ok_{\P(E)}(1)$ will be positive as well (see e.g. \cite{Z}, Example~7.10). Before we can apply techniques from pluripotential theory to define the push-forward in Proposition \ref{prop:Segre} for singular metrics, we need to check that $e^{-\varphi}$ is positive in the singular setting as well.

\begin{lma} \label{lma:inducedPositive}
Let $E$ be a vector bundle, and let $h$ denote a singular, Griffiths positive, hermitian metric on $E$. Let $e^{-\varphi}$ be the induced metric on the line bundle $\Ok_{\P(E)}(1)$. Then $e^{-\varphi}$ is a singular positive metric.
\end{lma}

\begin{proof}
The Griffiths positivity of $h$ implies (by definition) that the dual metric $h^*$ is Griffiths negative on $E^*$. By arguing as in \cite{PT}, Lemma 2.3.2 (a) the pullback metric $\pi^* h^*$ on $\pi^* E^*$ will also be Griffiths negative,  and since negativity is preserved when restricting to subbundles, this in turn implies that the induced metric on $\Ok_{\P(E)}(-1)$ is negative as well. Thus, the metric $e^{-\varphi}$ on $\Ok_{\P(E)}(1)$ is positive.
\end{proof}

The standard way of defining powers of currents of the form $dd^cu$, where $u$ is a plurisubharmonic function, is through the inductive approach of Bedford-Taylor \cite{BT} (here $d^c=\frac{i}{4\pi}(\dbar-\partial)$). However, in our case there are some extra twists. First off, for the classic theorem of Bedford-Taylor one needs the plurisubharmonic functions to be locally bounded, which is not the case for us. Fortunately, it turns out that this assumption can be replaced by requiring the \textit{unbounded locus} to be of appropriate codimension. Let us quickly recall this important concept.

\begin{df}\label{df:ULocus}
Let $u$ be a plurisubharmonic function on a complex manifold $X$. The \textit{unbounded locus} of $u$, $L(u)$, is defined to be the set of points $x\in X$ such that $u$ is unbounded in every neighborhood of $x$.
\end{df}

Note that if $h = e^{-\varphi}$ is a singular positive hermitian metric on a line bundle, then in a local frame, a local representative of $\varphi$ is plurisubharmonic, so $L(\varphi)$ is defined locally with respect to this frame. However, since two different local representatives of $\varphi$ as defined with respect to two different frames will differ by a locally bounded (harmonic) function, $L(\varphi)$ is well-defined, independent of the local frame.

\begin{remark}\label{remark:UEx}
Although $L(u)$ is always closed and contains the closure of the pole set of $u$, these two set are different in general. In \cite{D2}, Chapter III.4, the function
$$u(z)=\sum\frac{1}{k^2}\log\Big(\big|z-\frac{1}{k}|+e^{-k^3}\Big)$$
is provided as an example of a function which is everywhere finite in $\C$, but with $L(u)=\{0\}$.\qed
\end{remark}

The requirements we impose in the main theorems are inspired by Demailly's variant of the Bedford-Taylor
theorem that makes it possible to define powers of currents $dd^cu$,
where $u$ is a not necessarily locally bounded plurisubharmonic function.

\begin{thm}\label{thm:PPT}
Let $u$ be a plurisubharmonic function, and let $T$ denote a positive $(q,q)$-current. If $L(u) \cap \supp T$ is contained in some
variety $V$ with $\codim(V)\geq k+q$, then there exists a well-defined closed positive $(k+q,k+q)$-current $(dd^cu)^k\wedge T$,
which can be defined locally as the limit of
$$ (dd^c u_\nu)^k \wedge T$$
for any sequence $\{u_\nu\}$ of smooth plurisubharmonic functions decreasing pointwise to $u$.
\end{thm}

We refer to \cite{D2}, Chapter III, Theorem 4.5 and Proposition 4.9, or \cite{D1}, Proposition 2.2.3 for a proof of (more general variants of) this theorem.

\begin{remark}\label{remark:codim}
Looking at Proposition~\ref{prop:Segre} with $\Phi=dd^c\varphi$ for some plurisubharmonic function $\varphi$, at first glance it might seem tempting to use Theorem~\ref{thm:PPT} to define $(dd^c\varphi)^{k+r-1}$. For such a direct application of the theorem, however, we would need to assume that $L(\varphi)$ is contained in some subvariety $W$ with
$$\codim(W)\geq k+r-1,$$
a requirement which is much too strong to be of any practical use, and which in view of Lemma \ref{lma:UnbLoc} does not follow from the
codimension requirements in Theorem~\ref{thm:ContThm} or Theorem~\ref{thm:main}. Instead, we will aim at defining the push-forward directly.
Thus, the main part of the proofs of these theorems will basically be an adjustment of the proof of Theorem~\ref{thm:PPT} for the push-forward.
\end{remark}

\begin{lma}\label{lma:UnbLoc}
    Let $E$ be a holomorphic vector bundle with a singular Griffiths positive metric $h$, and let
    $e^{-\varphi}$ denote the singular positive metric on $\Ok_{\P(E)}(1)$ as in Lemma~\ref{lma:inducedPositive}.
    Then,
    \begin{equation} \label{eq:unbounded}
        L(\varphi) \subseteq \pi^{-1}(L(\log \det h^*))
    \end{equation}
    where $h^*$ denotes the induced dual metric on the dual bundle $E^*$.
\end{lma}

\begin{proof}
We take $x\in L(\log\det h^*)^c$ so that $\log\det h^*$ is bounded in
a neighborhood $U_x$ of $x$. We are going to show that the eigenvalues $\lambda_1,\ldots,\lambda_r$
of $h^*$ are uniformly bounded from above and below in $U_x$, by strictly positive constants.
This will show that $\pi^{-1}(x)\in L(\varphi)^c$ since if $e$ is a local frame for $\Ok_{\P(E)}(-1)$,
then $\varphi=\log\|e\|^2_{\pi^*h^*}$.

If $h^*$ is expressed as a matrix with respect to a holomorphic frame $u_1,\ldots,u_r$ for $E^*$, then
$$\tr(h^*)=\|u_1\|^2_{h^*}+\cdots+\|u_r\|^2_{h^*}.$$
By the Griffiths negativity of $h^*$, all the norms are plurisubharmonic, and so in particular (after
possibly shrinking $U_x$) bounded from above on $U_x$. Let $M>0$ be such that $\lambda_1,\ldots,\lambda_r\leq M$ on $U_x$.

On the other hand, $x\in L(\log\det h^*)^c$ implies that $\det h^*>C$ on $U_x$, for some $C>0$. Thus,
the smallest eigenvalue $\lambda_k$ of $h^*$ is uniformly bounded from below on $U_x$ by
$$\lambda_k>\frac{C}{\lambda_1\cdots\hat{\lambda}_k\cdots\lambda_r}\geq\frac{C}{M^{r-1}}>0.$$
\end{proof}

As mentioned above, the proofs of our main theorems will be an adaption of the proof of Theorem~\ref{thm:PPT}. The basic idea behind this proof is to reduce to the locally bounded situation, where one can apply the original Bedford-Taylor result. We will use the same basic approach in our proofs, but once the reduction has been carried out we will be in the setting of singular metrics on line bundles, and so we will need the following variant of the Bedford-Taylor theorem.

\begin{thm}\label{thm:PPT2}
Let $e^{-\varphi}$ be a non-degenerate singular positive metric on a line bundle $L$, let
$e^{-\varphi_\nu}$ be a sequence of smooth positive metrics on $L$, and let $e^{-\psi}$ be a fixed smooth metric on $L$. Let furthermore $\{ T_\nu \}$ be a sequence of closed positive currents
converging to a current $T$. Assume that either:\vspace{0.1cm}\\
\noindent(i) $\varphi_\nu$ converges locally uniformly to $\varphi$, or \\
\noindent(ii) $\varphi_\nu$ decreases pointwise to $\varphi$ and $T_\nu = T$.\vspace{0.1cm}

Then, the limits of
$$ (\varphi_\nu-\psi) (dd^c \varphi_\nu)^k \wedge T_\nu \text{ and } (dd^c \varphi_\nu)^k \wedge T_\nu,$$
as $\nu \to \infty$ exist in the sense of currents. These limits are independent of the sequences
$\{ \varphi_\nu \}$ and $\{ T_\nu \}$, and can thus be used as definitions of the $(k+q,k+q)$-currents
$(\varphi-\psi)(dd^c \varphi)^k\wedge T$ and $(dd^c \varphi)^k\wedge T$.
\end{thm}

A proof of this in the case when $L$ is a trivial line bundle and $\psi = 0$ can be found in for example \cite{D2},
Chapter III, in Corollary 3.6 under the assumption (i), and in Theorem 3.7 under the assumption (ii).
The case of singular metrics on line bundles then follows directly from this case since
in a local frame, $\varphi_\nu - \psi$ and $\varphi - \psi$ are the sum of a plurisubharmonic and a smooth function, whose sum is independent of the local trivialization, and in addition $dd^c \varphi_\nu$ is independent of the local trivialization.

\section{Segre Currents}
\label{sec:Segre_currents}
\noindent As mentioned previously, the existence of Chern currents will be achieved by expressing them in terms of Segre currents, as we do for forms in section~\ref{sec:segre-forms}. Hence, we will first need to prove the existence of such Segre currents and furthermore that it is possible to define the wedge products of these. A key result for achieving this is the fact that Definition~\ref{df:Gr} provides us with a notion of positivity that is easy to approximate, see \cite{BP}, Proposition 3.1 and \cite{R}, Proposition 6.2.

\begin{prop}\label{prop:appr}
Let $E$ be a trivial holomorphic vector bundle over a polydisc, and let $h$ be a singular Griffiths positive (negative) hermitian metric on $E$. Then on any smaller polydisc, there exists a sequence of smooth hermitian metrics $\{h_\varepsilon\}$  with positive (negative) Griffiths curvature, increasing (decreasing) to $h$ pointwise.
\end{prop}

The key ingredient in proving the local existence of Segre currents, and their wedge products, is Lemma \ref{lma:locallyBounded} below. Before we get to this result however, we first need to introduce the following concept:

\begin{df}
We say that a smooth $(p,p)$-form $\beta$ is a \emph{bump form} at a point $x$ if it is strongly positive, and such that for some (or equivalently for any) K\"ahler form $\omega$ defined near $x$, there exists a constant $C > 0$ such that $C\omega^p \leq \beta$ as strongly positive forms in a neighborhood of $x$.
\end{df}

We will also use in the proof of the following lemma, as well as later on, the basic formula
that if $\pi : X \to Y$ is a proper submersion, and if $\gamma$ is either a smooth
form or a current on $Y$, and $\eta$ is a smooth form on $X$, then 
\begin{equation} \label{eq:gammaeta}
    \gamma \wedge \pi_* \eta = \pi_*( \pi^* \gamma \wedge \eta).
\end{equation}

\begin{lma}\label{lma:locallyBounded}
Let $E$ be a trivial, rank $r$ vector bundle over a polydisc $P\subset\C^n$, let $h$ denote a singular Griffiths positive hermitian metric on $E$, and let $\{ h_\varepsilon \}$ be a family of smooth hermitian metrics on $E$. Let furthermore $\{T_\varepsilon\}$ be a sequence of closed positive $(q,q)$-currents converging to a current $T$.
Assume that $L(\log\det h^*)$ is contained in some subvariety $V$ of $P$ such that $\codim(V)\geq k+q$. Assume also that either:\vspace{0.1cm}\\
\noindent(i) Outside of $V$, $h$ is non-degenerate and continuous and $h_\varepsilon \to h$ locally uniformly, and on all of $P$, $T_\varepsilon \to T$, or \\
\noindent (ii) On all of $P$, $h_\varepsilon$ increases pointwise to $h$ and $T_\varepsilon = T$.

Then, for any point in $P$, there exists an $(n-k-q,n-k-q)$ bump form $\beta$ with arbitrarily small support
such that the limit of
\begin{equation} \label{eq:locallyBounded}
(-1)^k \int T_\varepsilon \wedge s_k(E,h_\varepsilon)\wedge \beta
\end{equation}
as $\varepsilon$ tends to 0 exists, and is independent of the sequences $\{h_\varepsilon\}$ and $\{T_\varepsilon\}$.
\end{lma}

\begin{proof}
The idea is to express the Segre forms $s_k(E,h_\varepsilon)$ as in Proposition~\ref{prop:Segre}, and then use a similar procedure
as the proof of Theorem~\ref{thm:PPT}.

First of all, by Lemma~\ref{lma:inducedPositive}, the Griffiths positivity of $h$ implies that the metric $e^{-\varphi}$ on $\Ok_{\P(E)}(1)$ is positive.
By \eqref{eq:unbounded}, if $W := \pi^{-1}(V)$, then $L(\varphi) \subseteq W$, and we note that since $\pi$ is a submersion,
$W$ has codimension $\geq k+q$ in $\P(E)$.

For a subset $I \subseteq \{1,\dots,n\}$ consisting of $n-k-q$
distinct elements $I_1 < \dots < I_{n-k-q}$, we define $p_I : \C^n \to \C^{n-k-q}$ to be the projection onto the
coordinates $(z_{I_1},\dots,z_{I_{n-k-q}})$, and similarly, we let $p_{I^c} : \C^n \to \C^{k+q}$ be the projection
onto the remaining coordinates. After any generic linear change of coordinates, for a subvariety $V$ of codimension $k+q$
with $0 \in V$, the projections $p_I$ and $p_{I^c}$ satisfy that the origin is an isolated point of $p_I^{-1}(0)$ in $V$,
and there exist small balls $B_I'\subset\C^{n-k-q}$ and $B_I''\subset\C^{k+q}$ and $0 < \delta_I < 1$  such that
$$V \cap p_I^{-1}(\bar B_I') \cap p_{I^c}^{-1}(\overline{B}_I'' \setminus (1-\delta_I) B_I '') = \emptyset.$$
Since this holds for a generic projection, we can in fact assume that it holds simultaneously for all
the projections $p_I$ and $p_{I^c}$ for all subsets $I \subseteq \{1,\dots,n\}$ consisting of $n-k-q$ elements.

We take $\chi_{I,1}$ and $\chi_{I,2}$ to be positive cut-off functions with compact support on $B_I'$ and $B_I''$
respectively, and which are both strictly positive at $0$, and such that $\chi_{I,2}$ is constant on $(1-\delta_I) B_I''$,
and we let
\begin{equation*}
    \beta := \sum_I (p_I^* \chi_{I,1}) (p_{I^c}^* \chi_{I,2}) idz_{I_1} \wedge d\bar{z}_{I_1} \wedge \dots \wedge i dz_{I_{n-k-q}} \wedge d\bar{z}_{I_{n-k-q}},
\end{equation*}
where the sum is over all ordered subsets $I = \{ I_1,\dots,I_{n-k-q} \mid I_1 < \dots < I_{n-k-q} \}$ of $\{1,\dots,n\}$. Then, $\beta$ is a bump form at $\{ 0 \}$,
and if $B_I'$ and $B_I''$ were chosen small enough above, then $\beta$ can be chosen to have arbitrarily small support.
By linearity, it is enough to prove \eqref{eq:locallyBounded} for each single term in the sum above, and without loss of generality,
we can assume that the term is when $I = \{1,\dots,n-k-q\}$. From now on, for convenience of notation, we write $\chi_1 = \chi_{I,1}$,
$\chi_2 = \chi_{I,2}$, $B' = B_I'$ and $B'' = B_I''$, $\delta = \delta_I$,
and write the coordinates on $\C^n$ as $(z',z'') \in \C^{n-k-q} \times \C^{k+q}$.
We will thus prove \eqref{eq:locallyBounded} for
\begin{equation} \label{eq:simpler_beta}
    \beta = \chi_1(z') \chi_2(z'') \beta_0, \text{ where } \beta_0 := i dz_1  \wedge d\bar{z}_1 \wedge \dots \wedge idz_{n-k-q} \wedge d\bar{z}_{n-k-q}.
\end{equation}

The approximating sequence $\{h_\varepsilon\}$ yields an approximating sequence $\{e^{-\varphi_\varepsilon}\}$ on $\Ok_{\P(E)}(1)$,
such that either $\varphi_\varepsilon$ converges locally uniformly to $\varphi$ outside of $W$ or decreases pointwise to $\varphi$, depending
on whether we are under assumption (i) or (ii). As these metrics are smooth, we can apply Proposition~\ref{prop:Segre}, and combining this with 
\eqref{eq:gammaeta}, we deduce that
$$(-1)^k\int_P T_\varepsilon \wedge s_k(E,h_\varepsilon) \wedge \beta =\int_P T_\varepsilon \wedge\pi_*\big((dd^c\varphi_\varepsilon)^{k+r-1}\big)\wedge  \beta
=\int_{P \times \P^{r-1}}\pi^*T_\varepsilon \wedge (dd^c\varphi_\varepsilon)^{k+r-1}\wedge \pi^* \beta,$$
where $\pi^* T_\varepsilon$ exists as a closed positive $(q,q)$-current since $\pi: P \times\P^{r-1}\to P$ is a proper submersion.

The main point with introducing $B'$ and $B''$ is that $V$ is disjoint from the compact cylinder,
$$K_\delta:=\overline{B}' \times \big(\overline{B}'' \setminus (1-\delta)B''\big)$$
and thus, $\varphi$ is locally bounded on $\pi^{-1}(K_\delta)$.
By performing integration by parts, we want to move to an integral on $\pi^{-1}(K_\delta)$, which then basically
reduces the problem to the well-known locally bounded setting.

Since we are dealing with metrics (not functions), we cannot directly perform integration by parts,
but we first need to add and subtract a smooth reference metric, say $(dd^c\varphi_1)^{k+r-1}$ with $\varepsilon=1$ fixed,
\begin{equation*}
\begin{split}
&\int_{P\times\P^{r-1}} \pi^*T_\varepsilon \wedge(dd^c\varphi_\varepsilon)^{k+r-1} \wedge \pi^* \beta=
\int_{P\times\P^{r-1}} \pi^*T_\varepsilon\wedge (dd^c\varphi_1)^{k+r-1} \wedge \pi^* \beta+\\
&\quad+\int_{P\times\P^{r-1}} \hspace{-1em}\pi^*T_\varepsilon\wedge dd^c(\varphi_\varepsilon-\varphi_1)\wedge \Big(\sum_{j=0}^{k+r-2}(dd^c\varphi_\varepsilon)^{k+r-2-j}\wedge(dd^c\varphi_1)^{j}\Big)\wedge  \pi^* \beta=:I+II.
\end{split}\end{equation*}
The integral $I$ converges as $\varepsilon \to 0$ since $T_\varepsilon \to T$. Since $\chi_1$ only depends on $z'$, while $\pi^*\beta$
is already of full degree in the $z'$-variables, we obtain that $d\chi_1 \wedge \beta_0 = 0$ and $d^c \chi_1 \wedge \beta_0 = 0$.
Using this in combination with that $\pi^* T_\varepsilon$ is $d$ and $d^c$-closed, we can formally
perform integration by parts in $II$, and obtain that it equals
\begin{equation*}
II = \int_{P \times\P^{r-1}}\pi^*T_\varepsilon \wedge(\varphi_\varepsilon-\varphi_1)\Big(\sum_{j=0}^{k+r-2}(dd^c\varphi_\varepsilon)^{k+r-2-j}\wedge(dd^c\varphi_1)^{j}\Big)\wedge\pi^*  \chi_1 \pi^*dd^c \chi_2 \wedge \pi^*\beta_0.
\end{equation*}

Since $\varphi$ is locally bounded on $\pi^{-1}(K_\delta)$, and  $\supp dd^c\chi_2 \subset \overline{B}''\setminus (1-\delta)B''$,
we get by Theorem~\ref{thm:PPT2} that the limit of $II$ when $\varepsilon$ tends to $0$
exists, and is independent of the choice of regularizing sequences $\{\varphi_\varepsilon\}$ and $\{T_\varepsilon\}$.
\end{proof}

From the proof, it follows that where $h$ is non-degenerate, one can form the product of a closed positive current and a Segre current.
More precisely:

\begin{lma}\label{lma:UniqueLim}
Let $E$, $h$, $\{h_\varepsilon\}$, $T$ and $\{T_\varepsilon\}$ be as in Lemma~\ref{lma:locallyBounded}.
If we are under either the assumptions (i) or (ii) of this lemma, then
outside of $L(\log \det h^*)$, the limit of
\begin{equation*}
    T_\varepsilon \wedge s_k(E,h_\varepsilon)
\end{equation*}
as $\varepsilon$ tends to $0$ exists in the sense of currents, and the limit is independent of the sequences $\{h_\varepsilon\}$ and $\{T_\varepsilon\}$.
\end{lma}

\begin{proof}
    If we fix $\varepsilon$, then $h_\varepsilon$ is smooth, and we denote by $e^{-\varphi_\varepsilon}$ the
    induced metric on $\Ok_{\P(E)}(1)$. Since $\varphi_\varepsilon$ is smooth, we get by Proposition~\ref{prop:Segre}
    together with \eqref{eq:gammaeta} that
    \begin{equation} \label{eq:UniqueLim}
        T_\varepsilon \wedge s_k(E,h_\varepsilon) = (-1)^k \pi_* ( \pi^* T_\varepsilon \wedge (dd^c \varphi_\varepsilon)^{k+r-1} ).
    \end{equation}
    By Theorem~\ref{thm:PPT2}, \eqref{eq:UniqueLim} has a limit as $\varepsilon$ tends to 0,
    and the limit is independent of the sequences $\{\varphi_\varepsilon\}$ and $\{T_\varepsilon\}$.
\end{proof}

As mentioned above, the local existence of Segre currents will follow from Lemma~\ref{lma:locallyBounded}. For the transition from local to global, we will use the following lemma.

\begin{lma}\label{lma:Unique}
Let $S$ and $T$ be two closed, positive $(k,k)$-currents such that $S = T$ outside a subvariety $A$ with $\codim(A)\geq k$, and assume that for each point of $p \in A$, there exists an $(n-k,n-k)$ bump form $\beta$ at $p$ with arbitrarily small support such that
\begin{equation*}
    \int S \wedge \beta = \int T \wedge \beta.
\end{equation*}
Then $S = T$ everywhere.
\end{lma}

\begin{proof}
The difference $U := S-T$ is a closed $(k,k)$-current of order $0$ with support on $A$. By \cite{D2}, Chapter III, Corollary 2.14, $U = \sum \lambda_i [A_i]$, where $A_i$ are the irreducible components of codimension $k$ of $A$, and $\lambda_i \in \mathbb{C}$. We want to show that $\lambda_i = 0$ for all $i$.

Fix some $i$, take a point in $A_i$ which does not belong to any of the other $A_j$'s, and take a $(n-k,n-k)$ bump form $\beta$ whose support still does not intersect any of the other $A_j$'s. Multiplying $S-T$ with $\beta$ and integrating, we get that
\begin{equation*}
    \lambda_i \int_{A_i} \beta = \int (S-T) \wedge \beta = 0.
\end{equation*}
The integral on the left-hand side is non-zero, so $\lambda_i = 0$.
\end{proof}

With these lemmas at our disposal, we can now define Segre currents and their wedge products.
\begin{prop} \label{prop:segre_products}
Let $E\to X$ be a rank $r$ holomorphic vector bundle over an $n$-dimensional complex manifold $X$, and let $h$ denote a singular Griffiths positive hermitian metric on $E$.
Take $k_1,\ldots,k_m\in\mathbb{N}$ and let $k := k_1 + \dots + k_m$. If $L(\log\det h^*)$ is contained in some subvariety $V$ of $X$ with
$\codim(V)\geq k$, then the wedge product
$$(-1)^k s_{k_1}(E,h)\wedge\cdots\wedge s_{k_m}(E,h)$$
can be defined as a closed, positive current on $X$, through regularization by an iterated limit as in (\ref{eq:IteratedLimit}).
\end{prop}

\begin{proof}
The proof is by induction over $m$, starting with the trivial case $m = 0$, when the current is simply $1$. We thus assume by induction that $T := (-1)^{k-k_m} s_{k_1}(E,h) \wedge \dots \wedge s_{k_{m-1}}(E,h)$ exists as a well-defined positive
closed current, and we start by proving existence locally, which we will achieve with the help of Lemma~\ref{lma:locallyBounded}. Hence, assume that $E$ is a trivial vector bundle over a polydisc, let $\{h_\varepsilon\}$ denote an approximating sequence to $h$ (as in Proposition~\ref{prop:appr}), and let $\{s_{k_m}^\varepsilon\}$ denote the corresponding sequence of Segre forms.
Since $(-1)^{k_m} s_{k_m}(E,h_\varepsilon)$ is a smooth and strongly positive $(k_m,k_m)$-form, the product $S_\varepsilon := (-1)^{k_m} T \wedge s_{k_m}^\varepsilon$ is a closed positive current, and its mass is dominated by its trace measure,
$$\| S_\varepsilon \| \leq \int S_\varepsilon \wedge (i\partial\dbar|z|^2)^{n-k}.$$
By Lemma~\ref{lma:locallyBounded}, this integral is uniformly bounded in $\varepsilon$ on a small enough neighborhood of $0$
since $(i\partial\dbar|z|^2)^{n-k}$ is dominated by a constant times any bump form at $\{ 0 \}$ in such a neighborhood.
Hence, it follows from the Banach-Alaoglu theorem that we can find a sequence $\varepsilon_j\rightarrow 0$ such that
$S_{\varepsilon_j}$ converges weakly to a closed, positive $(k,k)$-current $S$.
From combining Lemma~\ref{lma:UniqueLim} and Lemma~\ref{lma:locallyBounded} (under the assumption (ii)) with Lemma~\ref{lma:Unique},
we obtain that the limit $S$ is unique and only depends on $T$ and $h$.

We then in fact obtain that the full sequence $\{ S_{\varepsilon} \}$ converges to $S$
as $\varepsilon \to 0$, since otherwise, we could find a sequence $\{ S_{\varepsilon_j}\}$ with $\varepsilon_j \to 0$,
and a test-form $\phi$ such that the limit of $\int S_{\varepsilon_j} \wedge \phi$ exists, but is not equal to $\int S \wedge \phi$.
After passing to a subsequence, we could assume that $\{S_{\varepsilon_j}\}$ is indeed convergent,
and obtain a contradiction to the uniqueness of the limiting current $S$.

We take the limiting current $S$ as the (local) \emph{definition} of $(-1)^k s_{k_1}(E,h) \wedge \dots \wedge s_{k_m}(E,h)$,
and since it is independent of the regularizing sequence, this locally defined product of Segre currents in fact defines a global
current on $X$.
\end{proof}

\begin{remark}\label{remark:wedge_product_of_currents}
    Since we will need a bit more general result in the next section, we note that the proof of Proposition \ref{prop:segre_products}
    actually gives the following statement:

\emph{
Let $T$ be a closed positive $(q,q)$-current, let $\{\theta_\eps\}$ be a sequence of closed smooth $(k,k)$-forms, and let $V$ be a variety of codimension at least $k+q$. Moreover, we assume that the following three conditions are satisfied: \\
(i) $T\w\theta_\eps$ are positive (which is the case for instance when $\theta_\eps$ are strongly positive). \\
(ii) For any point in $X$, there exists a bump form $\beta$ such that $\int\! T\w\theta_\eps\w\beta$ converges. \\
(iii) Outside of $V$, $T\wedge \theta_\eps$ converges to a (closed positive) current $\tilde S$. \\
Then, the limit $S := \lim_{\eps\to 0} T\w\theta_\eps$ exists and is a closed positive current with $S|_{X\setminus V}=\tilde S$. Furthermore, $S$ is uniquely determined by $\tilde S$ in (iii) and the limits of $\int\! T\w\theta_\eps\w\beta$ in (ii).
}
\end{remark}

Theorem~\ref{thm:main} is an immediate consequence of Proposition~\ref{prop:segre_products} as outlined in the introduction. It is however not enough for Theorem~\ref{thm:ContThm}. For this latter theorem we need the following lemma, whose proof
is similar to the proof of Proposition~\ref{prop:segre_products}.

\begin{lma}\label{lma:UniqueLimCh}
Let $E\to X$ be a rank $r$ holomorphic vector bundle over an $n$-dimensional complex manifold $X$,
and let $h$ denote a singular Griffiths positive hermitian metric on $E$. Let also $\{ T_\varepsilon \}$
be a family of closed positive $(q,q)$-currents on $X$ converging to a current $T$.
Assume that $L(\log \det h^*)$ is contained in some subvariety $V$ of $X$ with $\codim(V) \geq k+q$.
If $\{ h_\varepsilon \}$ is a family of smooth hermitian metrics on $E$ converging locally uniformly to $h$ outside of $V$,
then there exist closed positive $(k+q,k+q)$ currents $S_{\varepsilon}^+$ and $S_{\varepsilon}^-$ such that
\begin{equation*}
   T_\varepsilon \wedge c_k(E,h_\varepsilon) = S_{\varepsilon}^+ - S_{\varepsilon}^-,
\end{equation*}
and the limits of both $S_{\varepsilon}^+$ and $S_{\varepsilon}^-$ as $\varepsilon$ tends to $0$ exist in the sense
of currents, and the limits are independent of the sequences $\{ h_\varepsilon \}$ and $\{ T_\varepsilon \}$.
\end{lma}

\begin{proof}
    By \eqref{eq:ChernSegre}, we get that $T_\varepsilon \wedge c_k(E,h_\varepsilon)$
    can be written as a finite sum of terms
    \begin{equation} \label{eq:induction_terms}
        (-1)^k a_K T_\varepsilon \wedge s_{k_1}(E,h_\varepsilon) \wedge \dots \wedge s_{k_m}(E,h_\varepsilon),
    \end{equation}
    where $K = (k_1,\dots,k_m) \in \mathbb{N}^m$ and $k_1 + \dots + k_m = k$, and $a_K$ is an integer.
    We let $S_{\varepsilon}^+$ be the sum of all terms \eqref{eq:induction_terms} with $a_K$ positive,
    and $S_{\varepsilon}^-$ the sum of the remaining terms.
    In order to prove that $S_{\varepsilon}^+$ and $S_{\varepsilon}^-$ have limits, which are independent of
    the sequences, it is thus enough to prove this for each term \eqref{eq:induction_terms}. The proof of this
    is by induction over $m$, and is essentially the same as the proof of local existence and uniqueness in
    the proof of Proposition~\ref{prop:segre_products}, only that the basic case $m=0$ here becomes the fact that
    $T_\varepsilon$ tends to $T$, and when applying Lemma~\ref{lma:locallyBounded} and Lemma~\ref{lma:UniqueLim},
    assumption (i) is used instead of assumption (ii).
\end{proof}

\begin{proof}[Proof of Theorem~\ref{thm:ContThm}]
    We want to prove that the limit of
    \begin{equation} \label{eq:chern_regularization}
        c_{k_1}(E_1,h^1_\varepsilon) \wedge \dots \wedge c_{k_m}(E_m,h^m_\varepsilon)
    \end{equation}
    as $\varepsilon$ tends to $0$ exists, and is independent of the regularizations $h^1_\varepsilon,\dots,h^m_\varepsilon$.

    In order to prove this, we will show by induction over $m$ that
    \begin{equation} \label{eq:induction_cont}
        c_{k_1}(E_1,h^1_\varepsilon) \wedge \dots \wedge c_{k_{m}}(E_{m},h^{m}_\varepsilon) = S_{m,\varepsilon}^+ - S_{m,\varepsilon}^-,
    \end{equation}
    where $S_{m,\varepsilon}^+$ and $S_{m,\varepsilon}^-$ are smooth closed positive forms, which have limits as currents as $\varepsilon$ tends to $0$,
    independent of the regularizing sequences of the metrics on $E_1,\dots,E_{m}$. The basic case $m = 0$ is trivial when we let $S_{0,\varepsilon}^+ = 1$
    and $S_{0,\varepsilon}^- = 0$. Thus by induction,  we assume that such $S_{m-1,\varepsilon}^+$ and $S_{m-1,\varepsilon}^-$ exist.
    Applying Lemma~\ref{lma:UniqueLimCh} with $T_\varepsilon$ either $S_{m-1,\varepsilon}^+$ or $S_{m-1,\varepsilon}^-$,
    and the Chern form $c_{k_m}(E_m,h^m_\varepsilon)$, we get in total four terms, two of which together define $S_{m,\varepsilon}^+$
    and two which define $S_{m,\varepsilon}^-$, and by Lemma~\ref{lma:UniqueLimCh}, these terms have limits
    as $\varepsilon$ tends to $0$, and which are independent of the regularizing sequences.
\end{proof}

\section{Cohomology Check}

\noindent
Locally, a positive singular hermitian metric $h$ on a vector bundle can be approximated by positive smooth metrics (see Lemma \ref{prop:appr}). On the projectivization of the vector bundle, this induces an approximation of the positive singular hermitian metric $e^{-\ph}$ induced by $h$
with smooth positive metrics $e^{-\ph_\eps}$ over preimages of small enough open sets. 
However, to prove the cohomology result (Theorem \ref{thm:Cohomology}), we need a global approximation of $e^{-\ph}$. In general, we do not have a global smooth approximation with positive metrics.
Instead, we have the following global regularization with smooth uniformly quasi-positive metrics (i.e.\ bounded from below by something negative but smooth) by Demailly: 

\begin{thm}[Main Theorem 1.1 in \cite{DemaillyReg}]\label{thm:Demailly-regularization}
Let $M$ be a compact hermitian manifold, and
let $e^{-\ph}$ be a singular positive metric on a line bundle $L$ over $M$.
Then, there exist a sequence of smooth metrics $\{e^{-\tpe}\}$ increasing to $e^{-\ph}$ 
and a smooth positive real $(1,1)$-form $u$ such that
	\[dd^c\tpe +u\geq 0.\]
\end{thm}

Actually, Demailly proves a much more general result but the version above is sufficient for our purpose. For the reader's convenience, let us show how the statement from above is deduced from Demailly's theorem. 

\begin{proof}
Let us use the same notation as in \cite{DemaillyReg}: We set $T:=dd^c \ph$ which is positive by assumption.
For a smooth metric $e^{-\ph_0}$ on $L$, the smooth real $(1,1)$-form $\alpha:=dd^c\ph_0$ is in the same $dd^c$-class as $T$. Furthermore,  $T=\alpha+dd^c\psi$ for the quasi-plurisubharmonic  function $\psi:=\ph-\ph_0$.
Let $g$ be a smooth hermitian metric on the tangent bundle $TM$, and let $e^{-\sigma}$ denote the associated metric on $\Ok_{\P(TM)}(1)$ over the projectivization of $TM$.
Although $g$ does not need to be positive, we know that $\sigma$ is positive with respect to tangent vectors along fibres of $\pi\colon\P(TM)\rightarrow M$. In fact,
if $\Theta:=\Theta(TM,g)$ denotes the curvature matrix of $g$, then 
\[dd^c \sigma(p,[v])=-\sfrac i{2\pi\cdot \|v\|^2_g} \langle\Theta v,v\rangle_g + \omega_{\textrm{FS},[g(p)]}\]
for $(p,[v])\in\P(TM)$ (see for instance \cite{G}*{Section 2}).
Let $\{U_j\}$ be an open Stein covering of $M$ such that $TM$ is trivial on $U_j$, and let $\chi_j$ be cut-off functions on $X$ with support in $U_j$ such that $\chi_j=1$ on slightly smaller open sets $\tilde U_j\subset U_j$. We may assume that $\{\tilde U_j\}$ still covers $M$.
On $U_j$, there exist strictly plurisubharmonic functions $a_j$ such that $\frac i\pi\langle\Theta v,v\rangle_g +dd^c a_j\geq 0$ for all $v\in TM$ with $\|v\|_g=1$. For $u:=\sum_j \chi_j dd^c a_j$, we get
	\[dd^c\sigma+u \geq -\sfrac i{2\pi\cdot \|v\|^2_g} \langle\Theta v,v\rangle_g +\omega_{\textrm{FS},[g(p)]} + dd^c\pi^\ast a_l\geq 0\]
on $\tilde U_l$ for every $l$.
Furthermore, we may assume that $u\geq\omega$ where $\omega$ denotes the hermitian form associated to $g$.
We conclude that all the assumptions of Main Theorem 1.1 in \cite{DemaillyReg} are satisfied.
That means that for $c>\max_{x\in M} \nu(T,x)$, there exist smooth quasi-plurisubharmonic functions $\psi_\eps:=\psi_{c,\eps}$ decreasing to $\psi$ such that
\[T_\eps:= \alpha+dd^c\psi_\eps \geq -c u- \omega.\]
Replacing $u$ by $(c+1) u$, we obtain that $T_\eps\geq -u$.
In particular, the smooth quasi-positive metrics $\tpe:=\ph_0+\psi_\eps$ on $L$ satisfy the claimed: $dd^c\tpe=\alpha+dd^c\psi_\eps =T_\eps\geq-u$ and $\tpe$ decreases pointwise to $\ph_0+\psi=\ph$.
\end{proof}

We would like to use Demailly's regularization although the approximating sequence $\{\tpe\}$ is only quasi-positive. For this, we need to consider slightly more general results than in section \ref{sec:Segre_currents} which are given by  the following two propositions. First, we consider a variation of Proposition \ref{prop:segre_products}:

\begin{prop} \label{prop:quasi_segre_products_1}
Let $E\to P$ be a trivial holomorphic vector bundle of rank $r$ over a polydisc $P\subset\C^n$, and
let $e^{-\psi}$ denote a singular positive metric on a holomorphic line bundle $L$ over $\P(E)$.
Let $\alpha$ be a closed positive $(1,1)$-form on $\P(E)$, and
let $T$ be a closed positive $(q,q)$-current on $P$.
If $\pi(L(\psi))$ is contained in some subvariety $V\subset P$ with $\codim(V)\geq k+q$, then for every $j=0{,}...,k+r-1$, 
there is a well-defined closed positive current
\begin{equation} \label{eq:alpha_segre_product}
    T\w\pi_\ast (\alpha^j\w(dd^c\psi)^{k+r-1-j}),
\end{equation}
which is defined as the limit of
\begin{equation*}
T\w\pi_\ast (\alpha^j\w(dd^c\psi_\eps)^{k+r-1-j}),
\end{equation*}
as $\eps$ tends to $0$, for any sequence $\{ e^{-\psi_\eps} \}$ of smooth positive metrics on $L$ increasing pointwise to $e^{-\psi}$.
\end{prop}

\begin{proof}
Replacing the form $s_k(E,h_\eps)$ by $\pi_\ast(\alpha^j\w(dd^c \ph_\eps)^{k+r-1-j})$, we see that the analogous statements of Lemma \ref{lma:locallyBounded} and Lemma \ref{lma:UniqueLim} (under the assumptions (ii)) hold true.
Since $\alpha^j\wedge (dd^c\psi_\eps)^{k+r-1-j}$ are strongly positive,
we get that
	\[T\w\pi_\ast (\alpha^j\wedge (dd^c\psi_\eps)^{k+r-1-j})=\pi_\ast\big(\pi^\ast T\w\alpha^j\wedge (dd^c\psi_\eps)^{k+r-1-j}\big)\]
are positive currents for all $\eps>0$.
Hence, the result follows from the statement in Remark~\ref{remark:wedge_product_of_currents} since the conditions (i--iii) in the remark are all satisfied.
\end{proof}

\begin{prop}\label{prop:quasi_segre_products_2}
Let $E\to P$ be a trivial  holomorphic vector bundle of rank $r$ over a polydisc $P\subset\C^n$, and let $e^{-\ph}$ denote a singular positive hermitian metric on $\Ok_{\P(E)}(1)$ induced by a positive hermitian metric on $E$.
Moreover, let $\alpha=dd^c a$, where $e^{-a}$ is a smooth positive metric on a holomorphic line bundle $L$ over $\P(E)$, and
let $T$ be a  closed positive $(q,q)$-current on $P$.
Assume that $\pi(L(\ph))$ is contained in some subvariety $V\subset P$ such that $\codim(V)\geq k+q$.
If $\{e^{-\tilde\ph_\eps}\}$ is a sequence of smooth metrics increasing pointwise to $e^{-\ph}$ and $dd^c\tilde\ph_\eps+\alpha\geq 0$,
then the limit of
	\[T\w\pi_\ast ((dd^c\tilde\ph_\eps)^{k+r-1})\]
    as $\eps$ tends to 0 exists and equals $T\w\pi_\ast ((dd^c\ph)^{k+r-1})$, as defined in \eqref{eq:alpha_segre_product}.
\end{prop}

\begin{proof} Let $l:=k+r-1$.
Let $\{ e^{-\ph_\eps} \}$ be a sequence of smooth positive metrics on $\Ok_{\P(E)}(1)$ increasing pointwise to $e^{-\ph}$, which exists since $\ph$ was induced by a positive metric on $E$, see Lemma \ref{prop:appr}.
Proposition \ref{prop:quasi_segre_products_1} (for $j=0$) then implies 
	\begin{equation}\label{eq:quasi-stuff_definition}
	T\w\pi_\ast ((dd^c\ph)^{l}) = \lim_{\eps\rightarrow 0} T\w\pi_\ast ((dd^c\ph_\eps)^{l}).
	\end{equation}
The proposition claims that the same holds when $\varphi_\eps$ is replaced by $\tpe$. To prove this, we account for the lack of positivity of
$\tpe$ by twisting:
We define $\psi:=\ph+a$, $\psi_\eps:=\tpe+a$ which give singular respectively smooth positive metrics $e^{-\psi}$ and $e^{-\psi_\eps}$ on $\Ok_{\P(E)}(1)\otimes L$. 
By assumption, $\psi_\eps$ decreases pointwise to $\psi$.
Thus, Proposition \ref{prop:quasi_segre_products_1} gives 
	\begin{equation}\label{eq:quasi-stuff_first_approximation}
	T\w\pi_\ast (\alpha^j\w(dd^c\psi)^{l-j}) = \lim_{\eps\rightarrow 0} T\w\pi_\ast (\alpha^j\w(dd^c\psi_\eps)^{l-j}).
	\end{equation}
On the other hand, $\{\ph_\eps+a\}$ is another sequence of positive smooth metrics on $\Ok_{\P(E)}(1)\otimes L$
which decreases to $\ph+a=\psi$. By Proposition \ref{prop:quasi_segre_products_1}, we get again:
	\begin{equation}\label{eq:quasi-stuff_second_approximation}
	T\w \pi_\ast (\alpha^j\w(dd^c\psi)^{l-j}) = \lim_{\eps\rightarrow 0} T\w\pi_\ast (\alpha^j\w(dd^c(\ph_\eps+a))^{l-j}).
	\end{equation}
Since
	\[(dd^c\ph_\eps)^l=(dd^c(\ph_\eps+a)-\alpha)^l=\sum{\Big.\!}_{j} \smath{\binom{l}{j}} (-\alpha)^j \w (dd^c(\ph_\eps+a))^{l-j},\]
we obtain: \newcommand{\widereq}{\makebox[2.2em]{=}}%
	\begin{equation*}\begin{split}
	T\w\pi_\ast ((dd^c\ph)^{l})
	&\overset{\eqref{eq:quasi-stuff_definition}}{\widereq} \lim_{\eps\rightarrow 0} T\w\pi_\ast ((dd^c\ph_\eps)^{l})\\
	&\widereq \sum{\Big.\!}_{j} \smath{\binom{l}{j}}  \lim_{\eps\rightarrow 0} T\w\pi_\ast \big((-\alpha)^j \w (dd^c(\ph_\eps+a))^{l-j}\big)\\
	&\overset{\eqref{eq:quasi-stuff_second_approximation}}{\widereq} \sum{\Big.\!}_{j} \smath{\binom{l}{j}}\ T\w\pi_\ast \big((-\alpha)^j \w (dd^c\psi)^{l-j}\big)\\
	&\overset{\eqref{eq:quasi-stuff_first_approximation}}{\widereq}	\sum{\Big.\!}_{j} \smath{\binom{l}{j}} \lim_{\eps\rightarrow 0} T\w\pi_\ast \big((-\alpha)^j \w (dd^c\psi_\eps)^{l-j}\big)\\
	&\widereq \lim_{\eps\rightarrow 0} T\w\pi_\ast ((dd^c\psi_\eps-\alpha)^{l}) \widereq \lim_{\eps\rightarrow 0} T\w\pi_\ast ((dd^c\tpe)^{l}).
	\qedhere
	\end{split}\end{equation*}%
\end{proof}
\medskip
\begin{proof}[Proof of Theorem~\ref{thm:Cohomology}]
The metric $e^{-\varphi}$ on $\Ok_{\P(E)}(1)$, induced by the singular metric $h$ on $E$, is positive by Lemma~\ref{lma:inducedPositive}.
Therefore, we can apply Theorem \ref{thm:Demailly-regularization} and obtain a smooth positive $(1,1)$ form $u$ and
a smooth sequence of metrics $\{ e^{-\tpe} \}$ increasing pointwise to $e^{-\ph}$ on $\P(E)$ such that $dd^c\tpe+u\geq 0$ for all $\eps>0$.
We set $T^\varepsilon_k:=(-1)^k\pi_*((dd^c\tilde\varphi_\varepsilon)^{k+r-1})$.
For all small enough open sets $U$ of $X$, 
there is a closed positive $(1,1)$-form $\alpha\geq u$ on $\pi^{-1}(U)$ such that $\alpha=dd^c a$ for a smooth positive metric $a$ on $\Ok_{\P(E|_U)}(\mu)$, $\mu\in\N$.
For instance, let $a$ be a $\mu$-fold tensor power (with $\mu$ big enough) of a (strictly) positive metric 
on $\Ok_{\P(E|_U)}(1)$ which is induced by a smooth positive metric on $E|_U$.
By an iterated application of Proposition \ref{prop:quasi_segre_products_2}, we obtain that
	\begin{equation}\begin{split}\label{eq:weak_converges_of_T_eps}
	\lim_{\varepsilon_m\to0}\cdots\lim_{\varepsilon_1\to0}T^{\varepsilon_1}_{k_1}\wedge\cdots\wedge T^{\varepsilon_m}_{k_m}
	&= (-1)^{k}\pi_\ast (dd^c\ph)^{k_1+r-1}\w\cdots\w\pi_\ast (dd^c\ph)^{k_m+r-1}\\
	&=s_{k_1}(E,h)\wedge\cdots\wedge s_{k_m}(E,h)
	\end{split}\end{equation}
for any positive integers $k_1,\ldots,k_m$ with $k_1+\cdots+k_m=k$.
Note that even though we apply Proposition~\ref{prop:quasi_segre_products_2} locally, \eqref{eq:weak_converges_of_T_eps} holds globally
on $X$ since all the currents are globally defined.

\smallskip

For an arbitrary smooth metric $h_0$ on $E$, let $e^{-\varphi_0}$ denote the induced metric on $\Ok_{\P(E)}(1)$. The Chern classes being invariants of $\Ok_{\P(E)}(1)$ implies that
$$[dd^c\tilde\varphi_\varepsilon]=[dd^c\varphi_0]$$
which yields 
$$[(dd^c\tilde\varphi_\varepsilon)^{j+r-1}]=[(dd^c\varphi_0)^{j+r-1}]$$
for all integers $j$.
Since the push-forward $\pi_*$ and exterior differentiation commute, this combined with Proposition \ref{prop:Segre} gives
$$[T^\varepsilon_j]=[s_j(E,h_0)]=s_j(E).$$
More generally, this exact same argument yields that
\begin{equation}\label{eq:Smooth_Cohomology}
[T^{\varepsilon_1}_{k_1}\wedge\cdots\wedge T^{\varepsilon_m}_{k_m}]=[s_{k_1}(E,h_0)\wedge\cdots\wedge s_{k_m}(E,h_0)].
\end{equation}
It remains to show that this implies that
$$[s_{k_1}(E,h)\wedge\cdots\wedge s_{k_m}(E,h)]=[s_{k_1}(E,h_0)\wedge\cdots\wedge s_{k_m}(E,h_0)].$$ 
This, however, is a straightforward consequence of \eqref{eq:weak_converges_of_T_eps}, \eqref{eq:Smooth_Cohomology} and Poincar\'e duality,
which yields that the exactness of $s_{k_1}(E,h)\wedge\cdots\wedge s_{k_m}(E,h)-s_{k_1}(E,h_0)\wedge\cdots\wedge s_{k_m}(E,h_0)$ is equivalent to showing that
$$\int_X\big(s_{k_1}(E,h)\wedge\cdots\wedge s_{k_m}(E,h)-s_{k_1}(E,h_0)\wedge\cdots\wedge s_{k_m}(E,h_0)\big)\wedge\beta=0$$
for all smooth, closed $2(n-k)$-forms $\beta$.
\end{proof}

\end{document}